\documentclass{article}

\usepackage{packages}

\author{Levente Bodnár}
\title{Generalized Turán Problem for Complete Hypergraphs}

\begin{document}

\maketitle

\begin{abstract}
    Write $K^{(k)}_{n}$ for the complete $k$-graph on $n$ vertices. For $2 \leq k \leq g < r$ integers, let $\pi\left(n, K^{(k)}_{g}, K^{(k)}_r\right)$ be the maximum density of $K^{(k)}_{g}$ in $n$ vertex $K^{(k)}_{r}$-free $k$-graphs. The main contribution of this paper is the upper bound: $\pi\left(n, K^{(k)}_{g}, K^{(k)}_r\right) \leq \left(1 + O\left(n^{-1}\right) \right)\prod_{m=k}^{g} \left(1 - \frac{\binom{m-1}{k-1}}{\binom{r-1}{k-1}} \right).$ The graph case ($k=2$) is the first known generalized Turán question, investigated by Erdős. The $k=g$ case is the hypergraph Turán problem where the best known general upper bound is by de Caen. The result proved here matches both bounds asymptotically, while any triple $k, g, r$ with $2 < k < g < r$ provides a new upper bound. The proof uses techniques from the theory of flag algebras to derive linear relations between different densities. These relations can be combined with linear algebraic methods. Additionally a simple flag algebraic certificate will be given for $\lim_{n \rightarrow \infty} \pi \left(n, K^{(3)}_4, K^{(3)}_5 \right) = 3/8$.
\end{abstract}

\section{Introduction}

\subsection{Hypergraph Turán Problems}

Given a $k$-graph $H$, write $\pi\left(n, H\right)$ for the maximum density of $k$-uniform edges among $H$-free hypergraphs with size $n$, and let $\pi \left(H \right) = \lim_{n \rightarrow \infty} \pi\left(n, H\right)$. It is known that the limit always exists. Let $K^{(k)}_n$ be the complete $k$-graph with $n$ vertices. A landmark result by Turán determined the values $\pi\left(n, K^{(2)}_r\right)$ exactly, with the unique graphs attaining the maximum.
\begin{theorem}[\cite{turan_original}]
$\pi\left(n, K^{(2)}_r\right)$ is uniquely attained at the balanced complete $(r-1)$-partite graph on $n$ vertices.
\end{theorem}
Following this, Erdős and Stone found more generally the value $\pi(H)$ for all graph $H$.
\begin{theorem}[\cite{erdos_stone}] Suppose $H$ is a graph with chromatic number $\chi(H)$, then 
\begin{equation*}
    \pi(H) = 1 - \frac{1}{\chi(H)-1}.
\end{equation*}
\end{theorem}

The corresponding question, when $k>2$, is still open and seems to be much more difficult. There are sporadic results for various $k$-graphs, but no $\pi\left(K^{(k)}_r\right)$ value is known. The best general upper bound comes from de Caen.
\begin{theorem}[\cite{de_caen_bound}]
\begin{equation*}
    \pi\left(n, K^{(k)}_r\right) \leq 1 - \left( 1 + \frac{r-k}{n-r+1} \right) \frac{1}{\binom{r-1}{k-1}}.
\end{equation*}
\end{theorem}
For an extensive survey, focusing on the $\pi\left(n, K^{(k)}_r\right)$ problem, with various lower and upper bounds, see \cite{sido_survey}. More recent coverage of the question with different $k$-graphs can be found in \cite{keevash_survey}.

\subsection{Generalized Turán Problems}

As a possible generalization of the Turán question, one can ask the maximum density of a given $k$-graph $F$,  instead of the $k$-edges. For $F, H$ given $k$-graphs, write $\pi(n, F, H)$ for the maximum density of $F$ among $n$ sized $H$-free $k$-graphs and use $\pi (F, H) = \lim_{n \rightarrow \infty} \pi(n, F, H)$. For complete graphs, this was initially investigated by Erdős \cite{erdos_generalized_1, erdos_generalized_2}.
\begin{theorem}[\cite{erdos_generalized_1}]
For $2 \leq g < r$ integers $\pi\left(n, K^{(2)}_g, K^{(2)}_r\right)$ is uniquely attained at the balanced complete $(r-1)$-partite graph on $n$ vertices.
\end{theorem}
Note that this gives asymptotically that $\pi\left(K^{(2)}_g, K^{(2)}_r\right) = \prod_{m=2}^{g} \left(1 - \frac{m-1}{r-1} \right)$. The generalized Turán problem for graphs was systematically investigated by Alon and Shikhelman, obtaining a result similar to Erdős-Stone.
\begin{theorem}[\cite{alon_generalized_erdos_stone}]
For any graph $H$, with chromatic number $\chi(H)$, the following holds 
\begin{equation*}
    \pi\left(K^{(2)}_g, H\right) = \prod_{m=2}^{g} \left(1 - \frac{m-1}{\chi(H)-1} \right).
\end{equation*}
\end{theorem}
In addition, \cite{alon_generalized_erdos_stone} investigates degenerate generalized Turán questions -- the rate of convergence of $\pi(n, F, H)$ when $\pi(F, H) = 0$. \cite{deg_gen_hyp_turan} finds various bounds for several degenerate generalized hypergraph Turán problems.

The generalized Turán problem for complete $k$-graphs corresponds with the separation of different layers of the boolean hypercube using a $k$-CNF. This idea appears for example in \cite{application_sidorenko} and will be further explored in a different paper. \cite{application_vcdim} gives new insights into the set of satisfying assignments of CNFs using a variant of the VC dimension. Bounds on this variant of the VC dimension turn out to be equivalent to a generalized Turán-type conjecture.

\subsection{Flag Algebras}
The theory of flag algebras \cite{razb_flag_algebras} provides a systematic approach to studying extremal combinatorial problems and the tools available for solving them. It gives a common ground for combinatorial ideas, by expressing them as linear operators, acting between flag algebras. Linearity means the different techniques can be easily combined with linear programming/linear algebra.

A large part of the theory can be automated with state-of-the-art optimization algorithms, providing spectacular improvements in density bounds. There has been significant progress in the famous tetrahedron problem \cite{de_caen_bound, tetra_previous} with the previous best bound being $\pi\left( K^{(3)}_4 \right) \leq \frac{3+\sqrt{7}}{12} < 0.59360$ while flag algebraic calculations improved it to the following bound: \begin{theorem}[\cite{razb_4_vertex} verified in \cite{flagmatic, baber_thesis}]
$\pi\left( K^{(3)}_4 \right) \leq 0.56167.$
\end{theorem} Note that the best known lower bound is $5/9 \leq \pi\left( K^{(3)}_4 \right)$. For excluded $K^{(3)}_5$ the calculations give the following: \begin{theorem}[\cite{baber_thesis}]
$\pi\left( K^{(3)}_5 \right) \leq 0.76954$
\end{theorem} with best known lower bound $3/4 \leq \pi\left( K^{(3)}_5 \right)$.

For a list of results provided by flag algebraic calculations, see \cite{razb_4_vertex, flagmatic}. The power of flag algebra has been illustrated in a wide range of other combinatorial questions \cite{ramsey_flag, permutation_flag}. Unfortunately, the computer-generated proofs lack insight and scale-ability compared to classical, hand-crafted arguments. They often only work in a small enough parameter range (for example, bounding $\pi\left(H\right)$ for $H$ with at most $7$ vertices). A survey by Razborov \cite{flag_interim_report} calls such applications plain. 

One of the goals of this paper is to show that the powerful plain flag algebra method, in this case, can be performed by hand, resulting in a general and scale-able theorem. The main ideas and proof steps, therefore, correspond with a plain application of flag algebra and were heavily inspired by it. During the proofs, relevant parts of the flag algebra theory will be highlighted. While the asymptotic result can be fully proved with flag algebraic manipulations, the bound with finite $n$ is only attainable with a more precise bounding of the errors. The theory is not explained here, for a quick introduction see \cite{flag_first_glance} or the original text \cite{razb_flag_algebras}.

\subsection{Overview of the Result}

In this paper, the generalized Turán problem for complete hypergraphs will be investigated, with the following contribution:
\begin{theorem}\label{main_theorem}
For integers $1 < k \leq g < r$ and any $n>(r-1) \left(1 + \left(\frac{(r-k)}{k-1}\right)^2 \right),$ 
\begin{equation*}
\begin{gathered}
     \pi\left(n, K^{(k)}_{g}, K^{(k)}_r\right) \leq \\
    \leq  \left(1 + \frac{(r-1) (r-k)^2}{(k-1)^2 n-(r-1) \left(2 k^2-2 k (r+1)+r^2+1\right)} \right)\prod_{m=k}^{g} \left(1 - \frac{\binom{m-1}{k-1}}{\binom{r-1}{k-1}} \right).
\end{gathered}
\end{equation*}
\end{theorem}

Note this means asymptotically that
\begin{corollary}
For integers $1 < k \leq g < r$ 
\begin{equation*}
    \pi\left(K^{(k)}_{g}, K^{(k)}_r\right) \leq \prod_{m=k}^{g} \left(1 - \frac{\binom{m-1}{k-1}}{\binom{r-1}{k-1}} \right).
\end{equation*}
\end{corollary}

The asymptotic bound is known to be tight when $k=2$, with the matching, balanced $(r-1)$-partite construction. Additionally, it agrees with the best-known general hypergraph Turán bound by de Caen \cite{de_caen_bound} which is conjectured to not be tight.

\cite{sido_survey} describes various lower bound constructions for the $2 < k = g$ case. A simple construction (when $k-1$ divides $r-1$) splits the vertex set into $\frac{r-1}{k-1}$ equal groups and includes each $k$ set that is not fully contained in a group. Using $l = \frac{r-1}{k-1}$ and an inclusion-exclusion calculation, this gives the asymptotic bound

\begin{equation*}
    \sum_{s=0}^{\lfloor g/k \rfloor} (-1)^{s} \binom{l}{s} \sum_{\substack{k \leq i_1, \dots, i_s \\ i_1+ \dots + i_s \leq g}} \binom{g}{i_1, \dots, i_s, g-i_1- \dots - i_s}l^{-i_1- \dots - i_s} \leq \pi\left( K^{(k)}_g, K^{(k)}_r\right).
\end{equation*}

While it is known that this construction is not optimal when $k=g$, in the $k \ll g$ regime, where $K^{(k)}_g$ appears more if the edges are "grouped", it provides a stronger bound.  \Cref{pi345_section} shows that this is asymptotically the best construction for $\pi \left( K^{(3)}_4, K^{(3)}_5\right)$. In general $g=r-1$ gives
\begin{equation*}
    \begin{split}
        & \ \sqrt{ 2 \pi r} \ \ e^{\frac{\log(2 \pi k)r}{2k}} \approx \\[2ex]
        \approx & \ \binom{r-1}{k-1, k-1, \dots , k-1} l^{-(r-1)} \leq \\[2ex]
        \leq & \ \pi\left(K^{(k)}_{r-1}, K^{(k)}_{r}\right) \leq \\[2ex]
        \leq & \ \prod_{m=k}^{r-1} \left( 1 - \frac{\binom{m-1}{k-1}}{\binom{r-1}{k-1}} \right) \leq \\[2ex]
        \leq & \ e^{(k-r)/k}.
    \end{split}
\end{equation*}

Given a hypergraph $G$, write $d(H, G)$ for the induced density of $H$ in $G$. The main tool used in the proof of \cref{main_theorem} is:
\begin{lemma}\label{main_lemma}
For all $0 \leq x$ and integers $k \leq m < n$, if $G$ is an $n$ vertex $k$-graph then
\begin{equation*}
    \begin{matrix*}[l]
        0 \geq & \left(- \frac{1 - \frac{k-1}{m}}{x}\right) & d\left(K^{(k)}_{m+1}, G\right) & + \\ & \left(2 - \frac{k-1}{mx} - \frac{1}{(n-m)x}\right) & d\left(K^{(k)}_m, G\right)  & + \\ & (-x) & d\left(K^{(k)}_{m-1}, G\right). &  
    \end{matrix*}
\end{equation*}
\end{lemma}
Note that the densities $d\left(K^{(k)}_m, G\right)$ only appear linearly in the expression. For different $k, g, r$ parameters, a convex combination of the expressions appearing in \cref{main_lemma}, with suitable $x, m$ values substituted in yields \cref{main_theorem}. As a comparison, \cite{de_caen_bound} utilizes similar ideas, but with a more complicated (non-linear) expression,
\begin{equation*}
    f_{m+1} \geq \frac{m^2 f_m}{(m-k+1)(n-m)} \left(\frac{f_m (n-m+1)}{f_{m-1} m} - \frac{(k-1)(n-m) + m}{m^2} \right)
\end{equation*} where $f_m = d\left(K^{(k)}_m, G\right)$ for short.

\subsection{Outline of the Paper}
\Cref{notation_section} summarizes the important notations and conventions throughout the paper. The proof of \cref{main_theorem} is included in  \cref{main_theorem_section} using two important components: \cref{main_lemma}, which is proved in \cref{lindens_section}; and a technical calculation (\cref{tridiag_lemma}), that is included in \cref{tridiag_section}. The short \cref{pi345_section} includes a certificate for $\pi\left(K^{(3)}_4, K^{(3)}_5\right) = 3/8$. The paper finishes with a few concluding remarks in \cref{outro_section}, the limitations of this approach and possible directions.

\section{Notation and Conventions}\label{notation_section}

\subsection{Basic Notation}

For a set $V$, the collection of subsets with size $k$ is denoted by $\binom{V}{k}$. The hypergraphs are identified with their edge sets; $G \subseteq \binom{V(G)}{k}$ is a $k$-graph with $V(G)$ vertex set. $K_n^{(k)}$ is the complete $k$-graph with $n$ vertices. For $S \subseteq V(G)$ the induced sub-hypergraph is $G \! \upharpoonright_S = G \cap \binom{S}{k}$. Hypergraph isomorphism is represented by  $G \simeq H$. $\mathcal{H}^{(k)}_n$ is the collection of non-isomorphic $k$-graphs having $n$ vertices.

Bold symbols indicate random variables. The uniform distribution from a set $V$ is represented by $\operatorname{Unif}(V)$. The density of $H$ in $G$ is defined to be $d(H, G) = \mathbb{P} \left[ G \! \upharpoonright_{\mathbf{S}} \simeq H \right]$ where $\mathbf{S} \sim \operatorname{Unif}\binom{V(G)}{|H|}$. Notice that this is the induced density, corresponding more with the flag algebraic approach, rather than the classical sub-hypergraph inclusion (referenced in the introduction). Write $d_s(H, G) = \mathbb{P} \left[ G \! \upharpoonright_{\mathbf{S}} \simeq F, H \subseteq F \right]$ for the classical inclusion with the same $\mathbf{S} \sim \operatorname{Unif}\binom{V(G)}{|H|}$. When $H$ is a complete hypergraph, the two notions are equivalent. The generalized Turán problem is to determine the value $$\pi\left(n, F, H\right) = \max\left\{d_s(F, G) \ : \ G \in \mathcal{H}_n, \quad d_s(H, G) = 0 \right\}.$$ The asymptotic problem asks $\pi(F, G) = \lim_{n \rightarrow \infty } \pi(n, F, G)$ (it is known that the limit always exists).

The quantity $$x_{m, r}^{(k)} = 1 - \frac{\binom{m-1}{k-1}}{\binom{r-1}{k-1}}$$ will be important, these are the terms appearing in the product.

\subsection{Flag Notation}
The hypergraphs represent the corresponding flags with empty type. $T^{(k)}_n$ is the complete type with $n$ vertices and all $k$-uniform edges. Type is indicated as a superscript. In particular $K^{(k), T^{(k)}_m}_{n}$ is the unique complete flag on $n$ vertices with a type having $m$ vertices. For a type $T$, $T$ also represents the flag with type $T$ and no extra vertices/edges. The averaging operator, transforming a $T$-typed flag $F^T$ into a flag with empty type is $\left\llbracket F^T \right\rrbracket_T$.

\subsection{Conventions}
In most of the proofs, the symbol $k$ is fixed, and the appearing statements concern $k$-graphs. For this reason, $k$ superscripts from the notations are often dropped. Additionally, $g, \ r, \ n$ symbols are reserved. They are integer parameters of the main question; determining the value of $\pi(n, K_g, K_r)$.

\section{Proof of Main Theorem}\label{main_theorem_section}

In this short section \cref{main_theorem} will be proved with the use of \cref{main_lemma} and \cref{tridiag_lemma}, a technical result included in \cref{tridiag_section}. Let's recall the main theorem, with the introduced $x^{(k)}_{m, r}$ notation.

\begin{reptheorem}{main_theorem}
Given integers $1 < k \leq g < r$ and $n>(r-1) \left(1 + \left(\frac{(r-k)}{k-1}\right)^2 \right)$, then 
\begin{equation*}
\begin{gathered}
    \pi\left(n, K^{(k)}_{g}, K^{(k)}_r\right) \\ \leq \left(1 + \frac{(r-1) (r-k)^2}{(k-1)^2 n-(r-1) \left(2 k^2-2 k (r+1)+r^2+1\right)} \right) \prod_{m=k}^{g} x^{(k)}_{m, r}.
\end{gathered}
\end{equation*}
\end{reptheorem}

In the following, the superscript $k$ is dropped from the notations for easier readability. The following claim gives bounds on the $x_{m, r}$ values. It follows easily after expanding the definition of $x_{m, r}$, therefore the proof is not included.
\begin{claim}\label{x_positive_claim}
    When $k-1 \leq m \leq r$, $$0 \leq x^{(k)}_{m, r} \leq 1,$$ with equality at $m=r$ and $m = k-1$ respectively. The smallest nonzero value is $x_{r-1, r} = \frac{k-1}{r-k}$
\end{claim}

\begin{proof}[Proof of \cref{main_theorem}]
Choose any $G \in \mathcal{H}_{n}$ (irrespective of the value of $d(K_{r}, G)$), with $$n>(r-1) \left(1 + \left(\frac{(r-k)}{k-1}\right)^2 \right),$$ and for short write $f_m = d\left(K_{m}, G\right)$. In the range $m \in \{k, k+1, ..., r-1\}$, $x_{m, r}$ is positive, therefore \cref{main_lemma}, with $x=x_{m, r}$ holds.
\begin{equation*}
    0 \geq - \frac{1 - \frac{k-1}{m}}{x_{m, r}}f_{m+1} + \left(2 - \frac{k-1}{m x_{m, r}} - \frac{1}{(n-m)x_{m, r}} \right) f_{m} - x_{m, r} f_{m-1}
\end{equation*}

The value $(n-m)x_{m, r}$ is minimal at $m=r-1$. Use $E_m$ for the above expression but $\frac{1}{(n-m)x_{m, r}}$ replaced with $\frac{r-k}{(n-r+1)(k-1)}$.
\begin{equation*}
    E_{m} = - \frac{1 - \frac{k-1}{m}}{x_{m, r}}f_{m+1} + \left(2 - \frac{k-1}{m x_{m, r}} - \frac{r-k}{(n-r+1)(k-1)} \right) f_{m} - x_{m, r} f_{m-1}
\end{equation*}

The replacement decreases the value, giving that each $E_{m}$ is still non-positive. This gives that any $\delta_{k}, \delta_{k+1}, ..., \delta_{r-1}$ sequence with all $0 \leq \delta_m$ results in $0 \geq \sum_{m=k}^{r-1} \delta_{m} E_{m}$. For a lower bound it is enough to find coefficients $0 \leq \delta_m$ satisfying \begin{equation} \label{main_theorem_equation}
    0 \geq \sum_{m=k}^{r-1} \delta_m E_{m} = - \delta_k x_{k, r} f_{k-1} + f_g - \delta_{r-1} \frac{1-\frac{k-1}{r-1}}{x_{r-1, r}} f_r.
\end{equation}
With the assumption that $f_r = d(K_r, G) = 0$ and the simple observation that $f_{k-1} = d(K_{k-1}, G) = 1$, one can deduce from \cref{main_theorem_equation} that $\pi\left(n, K_g, K_r \right) \leq \delta_k x_{k, r}$. Notice that finding $\delta_m$ corresponds with solving \cref{main_theorem_equation}, a system of linear equations. The technical \cref{tridiag_section} includes a way to approximate this system of linear equations. The following lemma summarizes the result, concluding the proof of \cref{main_theorem}.

\begin{lemma}\label{tridiag_lemma}
If $n>(r-1) \left(1 + \left(\frac{(r-k)}{k-1}\right)^2 \right)$ then the solution to the linear equations \cref{main_theorem_equation} satisfies that $\delta_m \geq 0$ and that $$\delta_k = \left(1 + \frac{(r-1) (r-k)^2}{(k-1)^2 n-(r-1) \left(2 k^2-2 k (r+1)+r^2+1\right)} \right) \prod_{m=k+1}^g x_{m, r}.$$
\end{lemma}
\end{proof}

\section{Linear Density Relations}\label{lindens_section}

This section covers the main combinatorial calculations involved in the proof of \cref{main_lemma}. The main purpose of \cref{main_lemma} is to act as a building block. Not only the validity is easier to verify, but it also involves the densities $d(K_m, G)$ linearly, therefore it can be combined easily; as illustrated in \cref{main_theorem}. The small claims in this section follow closely core results from the flag algebra theory. The connection will be highlighted in \cref{flag_claims_remark}. The connection between the plain flag algebra application and this proof is discussed in \cref{flag_sdp_remark}.

In the upcoming proofs, the $k$ superscript will not be included. $S$ is any subset of $V(G)$, while $S_m$ is an $m$ element subset of $V(G)$. Write $q(S)$ for the indicator function that is $1$ when $S$ induces a complete hypergraph in $G$ and $0$ otherwise. $l(S)$ is the number of $v \in V(G) \setminus S$ where $S+v$ is complete in $G$. The corresponding probability is $r(S_m) = \frac{l(S_m)}{n-m}$. Similarly $rr(S)$ is the probability that two different vertex extensions are both complete. $$rr(S_m) = \frac{\binom{l(S_m)}{2}}{\binom{n-m}{2}}$$ 

When $H \in \mathcal{H}_n$, then $s(H)$ is used for the size of the intersection of non-edges in $H$. In particular $s(K_n) = n$ and $s(K_n^-) = k$ where $K_n^-$ represents the hypergraph on $n$ vertices that has exactly one edge missing.

In the proof of \cref{main_lemma}, the fact, that $q(S)(r(S)-x)^2$ is always positive, will be exploited. Understanding the terms in the square is done by the following short claims. First, $r(S)^2$ and $rr(S)$ are related.

\begin{claim}\label{ppp_claim}
$$r(S_{m-1})^2 \leq rr(S_{m-1}) + r(S_{m-1}) \frac{1}{n-m}$$
\end{claim}

This simply follows from \begin{equation*}
r(S_{m-1})^2 - rr(S_{m-1}) = r(S_{m-1}) \left(\frac{l(S_{m-1})}{n-m+1} - \frac{l(S_{m-1})-1}{n-m}\right) \leq r(S_{m-1}) \frac{1}{n-m}.
\end{equation*} 

Second, linear equality between densities is shown.

\begin{claim}\label{chain_rule_claim}
Suppose $m \leq l \leq n$ with $F \in \mathcal{H}_{m}$ and $G \in \mathcal{H}_n$ then
\begin{equation*}
    d(F, G) = \sum_{H \in \mathcal{H}_{l}} d(F, H) d(H, G), 
\end{equation*}
in particular
\begin{equation*}
    d(K_m, G) = \sum_{H \in \mathcal{H}_{m+1}} \frac{s(H)}{m+1} d(H, G).
\end{equation*}
\end{claim}

\begin{proof}
Note that a uniform $m$ sized subset of $V(G)$ can be sampled by first choosing $\mathbf{S}_l \sim \operatorname{Uniform}\binom{V(G)}{l}$ and then $\mathbf{S}_m \sim \operatorname{Uniform}\binom{\mathbf{S}_l}{m}$. The claim follows from the law of total probability. The events $\{G \! \upharpoonright_{\mathbb{S}_l} \simeq H \ : \ H \in \mathcal{H}_l\}$ partition the probability space, therefore \begin{equation*}
\begin{split}
    d(F, G) = & \mathbb{P}\Big[ G \! \upharpoonright_{\mathbf{S}_m} \simeq F \Big] \\
    = & \sum_{H \in \mathcal{H}_l} \mathbb{P}\Big[ G \! \upharpoonright_{\mathbf{S}_m} \simeq F \ \Big| \ G \! \upharpoonright_{\mathbf{S}_l} \simeq H \Big] \mathbb{P}\Big[ G \! \upharpoonright_{\mathbf{S}_l} \simeq H \Big] \\
    = & \sum_{H \in \mathcal{H}_l} d(F, H) d(H, G).
\end{split}
\end{equation*}
The special case follows from $d(K_m, H) = \frac{s(H)}{m+1}$ when $H \in \mathcal{H}_{m+1}$.
\end{proof}

In the proof of \cref{main_lemma}, $S_{m-1}$ is chosen uniformly from the possible $m-1$ sized sets. The final claim connects the expected values arising in the terms of $q(S)(p(S_{m-1})-x)^2$ with densities of various $k$-graphs in $G$.

\begin{claim}\label{exp_claim} Suppose $\mathbf{S}_{m-1}$ is chosen uniformly randomly from the set $\binom{V(G)}{m-1}$ then
\leavevmode
\begin{enumerate}
    \item $$\mathbb{E} \left[ q(\mathbf{S}_{m-1})r(\mathbf{S}_{m-1}) \right] = d\left(K_{m}, G\right), $$
    \item $$\mathbb{E} \left[ q(\mathbf{S}_{m-1}) rr(\mathbf{S}_{m-1}) \right] = \sum_{H \in \mathcal{H}_{m+1}} \frac{\binom{s(H)}{2}}{\binom{m+1}{2}} d\left(H, G\right).$$
\end{enumerate}
\end{claim}

\begin{proof}
Similar to the proof of \cref{chain_rule_claim}, first choosing $\mathbf{S}_l \sim \operatorname{Uniform}\binom{V(G)}{l}$ and then $\mathbf{S}_{m-1} \sim \operatorname{Uniform}\binom{\mathbf{S}_l}{m-1}$ results in uniformly distributed $\mathbf{S}_{m-1}$. Note that $r(\mathbf{S}_{m-1})$ and $rr(\mathbf{S}_{m-1})$ corresponds with choosing $1$ and $2$ additional vertices accordingly, and then checking a condition on the extended set. 

In particular, for $\mathbf{S}_m, \mathbf{S}_{m-1}$ pair, let $R\left(\mathbf{S}_m, \mathbf{S}_{m-1}\right)$ be the event that $G \! \upharpoonright_{\mathbf{S}_m} \simeq K_m$ (and therefore $G \! \upharpoonright_{\mathbf{S}} \simeq K_{m-1}$). The claim follows from the law of total expectation, by conditioning on the shape of $G \! \upharpoonright_{\mathbf{S}_m}$
\begin{equation*}
\begin{split}
    & \mathbb{E} \left[ q(\mathbf{S}_{m-1})r(\mathbf{S}_{m-1}) \right] \\
    & \qquad = \sum_{H \in \mathcal{H}_{m}} \mathbb{P}\left[ R(\mathbf{S}_m, \mathbf{S}_{m-1}) \ \large| \ G \! \upharpoonright_{\mathbf{S}_m} = H \right] \mathbb{P}\left[ G \! \upharpoonright_{\mathbf{S}_m} = H \right] \\
    & \qquad = d(K_m, G).
\end{split}
\end{equation*}
Since only $H = K_{m}$ contains a suitable $m-1$ sized subset that satisfies $R_{\mathbf{S}_m, \mathbf{S}_{m-1}}$. 

For the second part, write $RR(\mathbf{S}_{m+1}, \mathbf{S}_{m-1})$ for the event that there are two copies of $K_{m}$ inside $G \! \upharpoonright_{\mathbf{S}_{m+1}}$ intersecting exactly at $\mathbf{S}_{m-1}$ (again this implies $G \! \upharpoonright_{\mathbf{S}_{m-1}} = K_{m-1}$). The calculation in this case gives
\begin{equation*}
\begin{split}
    & \mathbb{E} \left[ q(\mathbf{S}_{m-1})rr(\mathbf{S}_{m-1}) \right] \\
    & \qquad = \sum_{H \in \mathcal{H}_{m+1}} \mathbb{P}\left[ RR(\mathbf{S}_{m+1}, \mathbf{S}_{m-1}) \ \large| \ G \! \upharpoonright_{\mathbf{S}_{m+1}} = H \right] \mathbb{P}\left[ G \! \upharpoonright_{\mathbf{S}_{m+1}} = H \right] \\ 
    & \qquad = \sum_{H \in \mathcal{H}_{m+1}} \frac{\binom{s(H)}{2}}{\binom{m+1}{2}} d\left(H, G\right)
\end{split}
\end{equation*}
since in a given $G \! \upharpoonright_{\mathbf{S}_{m+1}} \simeq H$, a randomly chosen $\mathbf{S}_{m-1}$ satisfies $RR(\mathbf{S}_{m+1}, \mathbf{S}_{m-1})$ with probability $\frac{\binom{s(H)}{2}}{\binom{m+1}{2}}$.
\end{proof}

\begin{remark}\label{flag_claims_remark}

The above claims all correspond to parts of the general flag algebra theory \cite{razb_flag_algebras}.
\begin{enumerate}
    \item \cref{chain_rule_claim} corresponds to the chain rule (Lemma 2.2). In the language of flags it gives $$K_m = \sum_{H \in \mathcal{H}_{m+1}} \frac{s(H)}{m+1} H.$$
    \item \cref{ppp_claim} corresponds to products (Lemma 2.3). It is more or less equivalent with $$p\left(K_m^{T_{m-1}}; G^{T_{m-1}}\right)^2 - p\left(K_m^{T_{m-1}}, K_m^{T_{m-1}}; G^{T_{m-1}}\right) = O(|G|^{-1}).$$
    \item \cref{exp_claim} corresponds to averaging (Theorem 2.5). It is a restatement of \begin{equation*}
         \left\llbracket K_m^{T_{m-1}} \right\rrbracket_{T_{m-1}} = K_m
    \end{equation*}
    and \begin{equation*}
         \left\llbracket \left( K_m^{T_{m-1}} \right)^2 \right\rrbracket_{T_{m-1}} = \sum_{H \in \mathcal{H}_{m+1}} \frac{\binom{s(H)}{2}}{\binom{m+1}{2}} H. 
    \end{equation*}
\end{enumerate} 

\end{remark}

With these claims, the main lemma follows easily. 

\begin{replemma}{main_lemma}
For all $x>0$ and integers $k \leq m < n$, if $G \in \mathcal{H}_n$ then the following holds
\begin{equation*}
    \begin{matrix*}[l]
        0 \geq & \left(- \frac{1 - \frac{k-1}{m}}{x}\right) & d\left(K_{m+1}, G\right) & + \\ & \left(2 - \frac{k-1}{mx} - \frac{1}{(n-m)x}\right) & d\left(K_m, G\right)  & + \\ & (-x) & d\left(K_{m-1}, G\right). &  
    \end{matrix*}
\end{equation*}
\end{replemma}

\begin{proof}

When $x \in \mathbb{R}$ the quantity $q(S)(r(S) - x)^2$ is always non-negative. Therefore choosing $\mathbf{S}_{m-1} \sim \operatorname{Uniform}\binom{V(G)}{m-1}$ the following is true
$$0 \leq \mathbb{E} \left[ q(\mathbf{S}_{m-1}) (r(\mathbf{S}_{m-1}) - x)^2 \right].$$ Expanding the terms and applying \cref{ppp_claim} gives
\begin{equation*}
    0 \leq \mathbb{E} \left[ q(\mathbf{S}_{m-1})rr(\mathbf{S}_{m-1}) +\left(\frac{1}{n-m} - 2x\right)q(\mathbf{S}_{m-1})r(\mathbf{S}_{m-1}) + x^2 q(\mathbf{S}_{m-1}) \right].
\end{equation*}
The substitution from \cref{exp_claim} yields the following expression, without expected values:
\begin{equation*}
    0 \leq \sum_{H \in \mathcal{H}_{m+1}} \frac{\binom{s(H)}{2}}{\binom{m+1}{2}} d\left(H, G\right) + \left(\frac{1}{n-m} - 2x\right) d\left(K_{m}, G\right) + x^2 d\left(K_{m}, G\right).
\end{equation*}
Notice that $s(H)$ is maximal on $K_{m+1}$, otherwise it is at most $k$. This observation gives
\begin{equation*}
    \begin{split}
        0 & \leq \frac{k-1}{m} \sum_{^{H \in \mathcal{H}_{m+1}}_{H \neq K_{m+1}}} \frac{s(H)}{(m+1)} d\left(H, G\right) \\ & \qquad + d\left(K_{m+1}, G\right) \\ & \qquad + \left(\frac{1}{n-m} - 2x\right) d\left(K_{m}, G\right) \\ & \qquad + x^2 d\left(K_{m-1}, G\right).
    \end{split}
\end{equation*}
Finally expanding $\frac{k-1}{m} d(K_{m}, G)$ using \cref{chain_rule_claim} results in
\begin{equation*}
    \begin{matrix*}[l]
        0 \leq & \left(1 - \frac{k-1}{m}\right) & d\left(K_{m+1}, G\right) & + \\ & \left(\frac{k-1}{m} + \frac{1}{(n-m)} - 2x\right) & d\left(K_m, G\right)  & + \\ & x^2 & d\left(K_{m-1}, G\right). &  
    \end{matrix*}
\end{equation*}

Since $x \geq 0$, note that \cref{main_lemma} is a $-1/x$ multiple of the above, and the proof is complete.

\end{proof}

\begin{remark}\label{flag_sdp_remark}
This lemma can be easily stated as \begin{equation*}
    0 \leq  \left(1 - \frac{k-1}{m}\right)  K_{m+1} + \left(\frac{k-1}{m} - 2x\right) K_m + x^2 K_{m-1}  
\end{equation*} in the language of flags. The proof uses the expansion of the simple square
\begin{equation*}
    0 \leq \left\llbracket \left( K_m^{T_{m-1}} - xT_{m-1} \right)^2 \right\rrbracket_{T_{m-1}}
\end{equation*} In a plain application of flag algebra, the computer finds a conic combination of squares, similar to the above expression. \Cref{main_lemma} provides squares in a form that is easy to handle later (they only involve a small number of $d(K_m, G)$ values). The following section shows that the target expression \begin{equation*}
    K_g \leq \prod_{m=k}^g x^{(k)}_{m, r} + c K_r
\end{equation*} for some $c$ constant, lies in the conic combination of the squares.
\end{remark}

\section{The Associated Tridiagonal Matrix}\label{tridiag_section}

\begin{definition}
Given $k < r$ integers, the problem has an associated tridiagonal matrix $D^{(k)}_r$ with entries $d_{l, m}$ indexed by the range $k \leq l, m < r$ \begin{equation*}
    d_{l, m} = \begin{cases}
    -x^{(k)}_{m, r} & \text{  if } \  l=m-1\\
    \\
    2 - \frac{k-1}{mx^{(k)}_{m, r}} & \text{  if } \  l=m \\
    \\
    -\frac{1 - \frac{k-1}{m}}{x^{(k)}_{m, r}} & \text{  if } \  l=m+1\\
    \\
    0 & \text{  otherwise}
\end{cases}
\end{equation*}
\end{definition}

Write $\epsilon = \frac{r-k}{(n-r+1)(k-1)}$ and notice that $D^{(k)}_{r} - \epsilon I$ has column values equal to the coefficients in $E_{m}$. With the new notation, recall \cref{tridiag_lemma}, which was used to finish the proof of \cref{main_theorem}.

\begin{replemma}{tridiag_lemma}
Let $\Delta^{(k)}_{r}(\epsilon) = \left(D^{(k)}_{r}  - \epsilon I \right)^{-1}$ be the inverse of the associated tridiagonal matrix, with entries $\delta_{m, g}(\epsilon)$. If $0 \leq \epsilon < \frac{k-1}{(r-1)(r-k)}$, then the values $\delta_{m, g}(\epsilon)$ are all positive, and $$\delta_{k, g}(\epsilon) \leq \frac{1}{1-\epsilon \frac{(r-1)(r-k)}{k-1}} \prod_{m=k}^g x_{m+1, r}$$
\end{replemma}

Notice that $\epsilon=0$ corresponds with the inverse of $D_r^{(k)}$, and the asymptotic question when $n \rightarrow \infty$. This section is devoted to the proof of \cref{tridiag_lemma}, but it is illuminating and helpful for the proof to first calculate the inverse of $D_r^{(k)}$.

\subsection{The inverse of \texorpdfstring{$D^{(k)}_r$}{Dk,r}}

\begin{lemma}\label{simple_tridiag_lemma}
Let $\Delta^{(k)}_{r} = \left(D^{(k)}_{r}\right)^{-1}$ be the inverse of the associated tridiagonal matrix with entries $\delta_{m, g}$. Then $\delta_{m, g}$ are all positive and $$\delta_{k, g} = \prod_{m=k+1}^{g} x^{(k)}_{m, r}.$$
\end{lemma}

In the following proofs, the $k$ superscripts are omitted to increase readability. The complete inverse can be calculated following the method described in \cite{tridiagonal_inverse}. The value $\theta_m$ represents the determinant of the rows and columns indexed by the set $\{k, k+1, ..., m\}$, while $\phi_m$ is the determinant for the rows and columns indexed by $\{m, m+1, ..., r-1\}$.

From the cofactor calculation of determinants and inverses, the following claim is easy to verify.
\begin{claim}\label{determinant_long_claim}\leavevmode
\begin{enumerate}
    \item For all $k \leq m < r$ the induction $$\theta_m = d_{m, m} \theta_{m-1} - d_{m-1, m} d_{m, m-1} \theta_{m-2}$$ holds with initial values $\theta_{k-1} = 1$ and $\theta_{k-2} = 0$
    \item For all $k \leq m < r$ the reverse induction $$\phi_{m} = d_{m, m} \phi_{m+1} - d_{m+1, m} d_{m, m+1} \phi_{m+2}$$ holds with initial values $\phi_{r} = 1$ and $\phi_{r+1} = 0$
    \item For all $k \leq m < r$ the determinant can be calculated $$\operatorname{Det}\left(D_r\right) = \theta_{r-1} = \phi_{k} = \theta_m \phi_{m+1} - d_{m, m+1} d_{m+1, m} \theta_{m-1} \phi_{m+2}$$
    \item The entries of the inverse matrix are \begin{equation*} 
    \delta_{m, g} =  \frac{(-1)^{m+g+r-k}}{\operatorname{Det}(D_r)}\begin{cases} 
    \theta_{m-1} \phi_{g+1} \prod_{i=m}^{g-1} d_{i, i+1} & \ \text{if } m \leq g \\ 
    \ \\
    \theta_{g-1} \phi_{m+1} \prod_{i=g}^{m-1} d_{i+1, i} & \ \text{otherwise} 
    \end{cases} 
    \end{equation*} In the corresponding $k \leq m, g <r$ range.
\end{enumerate}
\end{claim}
First the $\phi_{m}$ values will be calculated using the recursive expression above.
\begin{claim}\label{phi_claim}
$\phi_{m} = 1$ in the range $k \leq m \leq r$
\end{claim}

\begin{proof}
By reverse induction. The claim holds for $m=r$ and notice $$\phi_{r-1} = d_{r-1, r-1} = 2 - \frac{\frac{k-1}{r-1}}{1- \frac{\binom{r-2}{k-1}}{\binom{r-1}{k-1}}} = 1$$ 
Using point 2 from \cref{determinant_long_claim} \begin{equation*}
\begin{split}
    \phi_{m} = & d_{m, m} - d_{m+1, m} d_{m, m+1} \\ =& 2 - \frac{k-1}{mx_{m, r}} - \left( 1 - \frac{k-1}{m}\right) \frac{x_{m+1, r}}{x_{m, r}} \\
     =& 2 - \frac{k-1}{m(1-u)} - \left( 1 - \frac{k-1}{m}\right) \frac{1-\frac{m
   u}{m-k+1}}{1-u} \\ =& 1
\end{split}
\end{equation*}
Where $u = \frac{\binom{m-1}{k-1}}{\binom{r-1}{k-1}} = 1-x_{m, r}$ and $u\frac{m}{m-k+1} = \frac{\binom{m}{k-1}}{\binom{r-1}{k-1}}= 1-x_{m+1, r}$ substitution was used to simplify the calculation.
\end{proof}
This gives that $\operatorname{Det}\left(D_r\right) = \theta_{r-1} = \phi_{k} = 1$.

\begin{claim}\label{theta_claim}
In the $k-1 \leq m < r$ range, $\operatorname{Sign}(\theta_m) = 1$
\end{claim}

\begin{proof}
By induction, note that the claim holds for $\theta_{k-1} = 1$. Then using point 3 from \cref{determinant_long_claim} and the value for the determinant,
\begin{equation*}
    \theta_m = 1 + d_{m, m+1}d_{m+1, m}\theta_{m-1}
\end{equation*}
Using \cref{x_positive_claim}, note that the values $d_{m, m+1} = -x_{m+1, r}$ and $d_{m+1, m} = -\frac{1 - \frac{k-1}{m}}{x_{m, r}}$ are both negative. Therefore their product; and by induction, $\theta_{m-1}$, are positive.
\end{proof}

\begin{proof}[Proof of \cref{simple_tridiag_lemma}]
It claims two things,
\begin{enumerate}
    \item The entries $\delta_{m, g}$ are all positive: \cref{determinant_long_claim} point 4 gives that \begin{equation*}
        \operatorname{Sign}(\delta_{m, g}) = (-1)^{m+g+r-k} \begin{cases} \prod_{i=m}^{g-1} \operatorname{Sign}(d_{i, i+1}) & \text{if } m \leq g \\ 
        \prod_{i=g}^{m-1} \operatorname{Sign}(d_{i+1, i}) & \text{otherwise.}\end{cases}
    \end{equation*} using $\operatorname{Sign}(\theta_m)=1$ from \cref{theta_claim} and that $\phi_m = 1$ from \cref{phi_claim}. Since all $d_{i, i+1}, d_{i+1, i}$ are negative, the inverse of $D_r$ has only positive entries.
    \item $\delta_{k, g} = \prod_{m=k}^{g} x_{m+1, r}$: Again substituting the values $d_{i, i+1} = -x_{i+1, r}$ and $\operatorname{Det}(D_r) = \phi_{g+1} = \theta_{k-1} = 1$ into \cref{determinant_long_claim} point 4 gives the stated value for $\delta_{k, g}$.
\end{enumerate}

\end{proof}

Interestingly, the values $\delta_{k, g}$ are easy enough to calculate exactly. In contrast, $\delta_{m, g}$ requires the value of some $\theta_m$ which is difficult to find in general with the recursive expression. The sign of $\theta_m$ is easy to find, exactly what is needed for the proof.

\subsection{The inverse of \texorpdfstring{$D^{(k)}_r - \epsilon I$}{Dk,r - eps I}}

Consider the same calculation but with $D_r - \epsilon I$. The value $\phi_m(\epsilon)$ is the determinant for the rows and columns indexed by $\{m, m+1, ..., r-1\}$ of $D_r - \epsilon I$. The next claim shows that $\phi_m(\epsilon)$ is increasing in $m$ when $\epsilon>0$. A notation for the increments will be useful, write $\zeta_m(\epsilon) = \phi_{m+1}(\epsilon) - \phi_m(\epsilon)$.

\begin{claim}
When $k \leq m < r$ and $0 \leq \epsilon \leq \frac{k-1}{(r-1)(r-m)}$, $$0 \leq \zeta_m(\epsilon) \leq \epsilon \frac{r-1}{k-1} \left(1 - \left( 1 - \frac{k-1}{r-1} \right)^{r-m}\right)$$ and correspondingly $$1 \geq \phi_m(\epsilon) \geq 1 - \epsilon \frac{(r-1)(r-m)}{k-1}$$
\end{claim}

\begin{proof}
Use point 2 from \cref{determinant_long_claim}. The initial value is $\zeta_{r}(\epsilon) = 0$ and by reverse induction take 
\begin{equation*}
    \begin{split}
        \phi_m(\epsilon) = & \left( d_{m, m} - \epsilon \right) \phi_{m+1}(\epsilon) - d_{m+1, m} d_{m+1, m} \phi_{m+2}(\epsilon) \\
         = & \left( 2 - \frac{k-1}{mx_{m, r}} - \epsilon \right) \phi_{m+1}(\epsilon) - \left( 1 - \frac{k-1}{mx_{m, r}} \right) \left( \phi_{m+1}(\epsilon) + \zeta_{m+1}(\epsilon) \right) \\
          = & \phi_{m+1}(\epsilon) - \zeta_{m+1}(\epsilon) \left(1 - \frac{k-1}{mx_{m, r}} \right) - y\phi_{m+1}(\epsilon). 
    \end{split}
\end{equation*}
Therefore \begin{equation}\label{zeta_equation}
    \zeta_m(\epsilon) = \zeta_{m+1}(\epsilon) \left(1 - \frac{k-1}{mx_{m, r}} \right) + \epsilon \phi_{m+1}(\epsilon).
\end{equation} 
Note that $$0 \leq d_{m, m+1} d_{m+1, m} = \left(1 - \frac{k-1}{mx_{m, r}} \right) \leq \left(1 - \frac{k-1}{r-1} \right)$$ in the $k \leq m < r$ range. As $\epsilon \leq \frac{k-1}{(r-1)(r-m)} < \frac{k-1}{(r-1)(r-m-1)}$ by reverse induction it holds that $0 \leq \phi_{m+1}(\epsilon)$ and $0 \leq \zeta_{m+1}(\epsilon)$ giving the required lower bound $0 \leq \zeta_{m}(\epsilon)$. This implies the upper bound $\phi_{m}(\epsilon) \leq 1$.

For the $\zeta_m(\epsilon)$ upper bound, in \cref{zeta_equation} bound each term: $m x_{m, r} \leq (r-1)$ and $\phi_{m+1}(\epsilon) \leq 1$. This gives the intermediate result $$\zeta_m(\epsilon) \leq \zeta_{m+1}(\epsilon) \left(1 - \frac{k-1}{r-1} \right) + \epsilon.$$ which, by iterated application and $\zeta_{r}(\epsilon) = 0$ initial value, implies $$\zeta_m(\epsilon) \leq \epsilon \frac{r-1}{k-1} \left(1 - \left( 1 - \frac{k-1}{r-1} \right)^{r-m}\right).$$ A summation formula for the upper and lower $\zeta_m(\epsilon)$ bounds combined with the initial $\phi_{r}(\epsilon) = 1$ value gives $$1 \geq \phi_m(\epsilon)\geq 1 - \epsilon \frac{(r-1)^2}{(k-1)^2} \left( \frac{(k-1)(r-m)}{r-1} + \left(1 - \frac{k-1}{r-1} \right)^{r-m} - 1\right).$$ This provides a tighter bound but for simplicity use $$1 - \epsilon \frac{(r-1)^2}{(k-1)^2} \left( \frac{(k-1)(r-m)}{r-1} + \left(1 - \frac{k-1}{r-1} \right)^{r-m} - 1\right) \geq 1 - \epsilon \frac{(r-1)(r-m)}{k-1}.$$
\end{proof}

This gives a simple linear bound for the determinant. When $0 < \epsilon < \frac{k-1}{(r-1)(r-k)}$, $$1 - \epsilon \frac{(r-1)(r-k)}{k-1} \leq \phi_k(\epsilon) = \operatorname{Det}(D_r - \epsilon I) \leq 1.$$ If $0 \leq \epsilon < \frac{k-1}{(r-1)(r-k)}$ then the determinant is strictly positive, bounding the smallest eigenvalue of $D_r$. 

\begin{proof}[Proof of \cref{tridiag_lemma}]
Again it claims two things.
\begin{enumerate}
    \item The entries $\delta_{m, g}(\epsilon)$ are all positive: By assumption, $\epsilon$ is smaller than the smallest eigenvalue of $D_r$. Therefore the expansion $$\left(D_r - \epsilon I \right)^{-1} = D_r^{-1} + \epsilon D_r^{-2} + \epsilon^2 D_r^{-3}... $$ holds. \Cref{simple_tridiag_lemma} shows that the entries in $D_r^{-1}$ (and in $D_r^{-i}$ for $1 > i$ correspondingly) are all positive, giving the required positivity of $\delta_{m, g}(\epsilon)$.
    \item $$\delta_{k, g}(\epsilon) \leq \frac{1}{1-\epsilon \frac{(r-1)(r-k)}{k-1}} \prod_{m=k}^g x_{m+1, r}:$$ This follows from substituting the bounds $\phi_{m}(\epsilon) \leq 1$ and $1 - \epsilon \frac{(r-1)(r-k)}{k-1} \leq \operatorname{Det}(D_r - \epsilon I)$ into \cref{determinant_long_claim} point 4. 
\end{enumerate}

\end{proof}

\section{\texorpdfstring{$\pi\left(K^{(3)}_4, K^{(3)}_5\right) = 3/8$}{pi K3,4, K3,5 = 3/8}}\label{pi345_section}
The combination of more sophisticated, but still simple squares can provide tight bounds for $\pi\left(K^{(3)}_4, K^{(3)}_5\right) = 3/8$. For an easier description of the flags, consider the complement question. If $E_n$ is the $3$-graph with $n$ vertices and no edges then $$\pi\left(K^{(3)}_4, K^{(3)}_5\right) = \lim_{n \rightarrow \infty} \max \left\{ d(E_4, G) \ : \ G \in \mathcal{H}^{(3)}_n, \ \ d(E_5)=0 \right\}.$$

Use $P_n$ for the corresponding type with $n$ vertices and no edges. Note $E_n^{P_m}$ is unique for any $m<n$ pair. Define the further flags:
\begin{enumerate}
    \item $L^{P_2}_{a}$ is the flag with vertex set $\{0, 1, 2, 3\}$, edge set $\{(0, 2, 3)\}$ and type formed from the vertices $0, 1$. Similarly, write $L^{P_2}_{b}$ for the flag with the same vertex set and type but $\{(1, 2, 3)\}$ edge set. 
    \item Use $M^{P_3}_a$ for the flag with vertices $\{0, 1, 2, 3\}$, edges $\{(1, 2, 3)\}$ and type from $0, 1, 2$. Symmetrically, with the same vertex and type set use $M^{P_3}_b$ for the edge set $\{(0, 2, 3)\}$. And $M^{P_3}_c$ for the edge set $\{(0, 1, 3)\}$.
    \item $N^{Q_4}$ is the flag with vertices $\{0, 1, 2, 3, 4\}$, edges $\{(0, 1, 2)\}$ and type formed by $0, 1, 2, 3$. Note that $Q_4$ does not agree with any of the $T_n$ or $P_n$ types.
    \item $O^{T_4}_{a}$ has vertex set $\{0, 1, 2, 3, 4\}$, edge set $\{(0, 1, 4)\}$ and type formed from $0, 1, 2, 3$. Additionally write $O^{T_4}_{b}$ for the flag with the same vertex and type set but $\{(2, 3, 4)\}$ edges.
\end{enumerate}

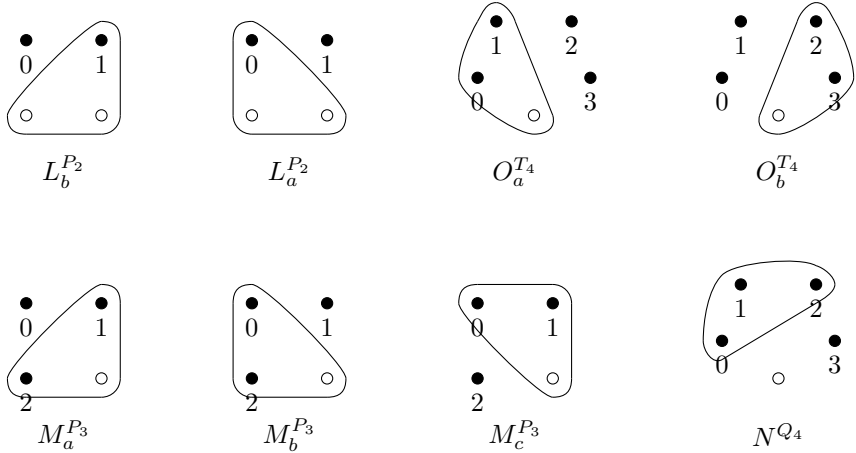
\begin{tikzpicture}[scale=0.5]
	\begin{pgfonlayer}{nodelayer}
		\node [style={type_vertex}, label={below:0}] (12) at (0, 2) {};
		\node [style={type_vertex}, label={below:1}] (13) at (2, 2) {};
		\node [style=vertex] (14) at (0, 0) {};
		\node [style=vertex] (15) at (2, 0) {};
		\node [style=none] (16) at (-0.5, 0) {};
		\node [style=none] (17) at (0, -0.5) {};
		\node [style=none] (18) at (2, -0.5) {};
		\node [style=none] (19) at (2.5, 0) {};
		\node [style=none] (20) at (2.5, 2) {};
		\node [style=none] (21) at (2, 2.5) {};
		\node [style=none] (39) at (1, -1.5) {$L_b^{P_2}$};
		\node [style={type_vertex}, label={below:1}] (40) at (8, 2) {};
		\node [style={type_vertex}, label={below:0}] (41) at (6, 2) {};
		\node [style=vertex] (42) at (8, 0) {};
		\node [style=vertex] (43) at (6, 0) {};
		\node [style=none] (44) at (8.5, 0) {};
		\node [style=none] (45) at (8, -0.5) {};
		\node [style=none] (46) at (6, -0.5) {};
		\node [style=none] (47) at (5.5, 0) {};
		\node [style=none] (48) at (5.5, 2) {};
		\node [style=none] (49) at (6, 2.5) {};
		\node [style=none] (50) at (7, -1.5) {$L_a^{P_2}$};
		\node [style={type_vertex}, label={below:0}] (51) at (12, 1) {};
	    \node [style={type_vertex}, label={below:1}] (52) at (12.5, 2.5) {};
		\node [style={type_vertex}, label={below:2}] (53) at (14.5, 2.5) {};
		\node [style={type_vertex}, label={below:3}] (54) at (15, 1) {};
		\node [style=vertex] (55) at (13.5, 0) {};
		\node [style=none] (56) at (13.5, -0.5) {};
		\node [style=none] (57) at (14, 0) {};
		\node [style=none] (58) at (11.5, 1) {};
		\node [style=none] (59) at (12.5, 3) {};
		\node [style=none] (60) at (12, 2.5) {};
		\node [style=none] (61) at (13, 2.5) {};
		\node [style={type_vertex}, label={below:2}] (62) at (12, -7) {};
		\node [style=none] (66) at (11.5, -5) {};
		\node [style=none] (67) at (12, -4.5) {};
		\node [style=none] (68) at (14, -4.5) {};
		\node [style=none] (69) at (14.5, -5) {};
		\node [style=none] (70) at (14.5, -7) {};
		\node [style=none] (71) at (14, -7.5) {};
		\node [style={type_vertex}, label={below:1}] (72) at (8, -5) {};
		\node [style={type_vertex}, label={below:0}] (73) at (6, -5) {};
		\node [style=vertex] (74) at (8, -7) {};
		\node [style=none] (76) at (8.5, -7) {};
		\node [style=none] (77) at (8, -7.5) {};
		\node [style=none] (78) at (6, -7.5) {};
		\node [style=none] (79) at (5.5, -7) {};
		\node [style=none] (80) at (5.5, -5) {};
		\node [style=none] (81) at (6, -4.5) {};
		\node [style={type_vertex}, label={below:0}] (82) at (0, -5) {};
		\node [style={type_vertex}, label={below:1}] (83) at (2, -5) {};
		\node [style=vertex] (85) at (2, -7) {};
		\node [style=none] (86) at (-0.5, -7) {};
		\node [style=none] (87) at (0, -7.5) {};
		\node [style=none] (88) at (2, -7.5) {};
		\node [style=none] (89) at (2.5, -7) {};
		\node [style=none] (90) at (2.5, -5) {};
		\node [style=none] (91) at (2, -4.5) {};
		\node [style={type_vertex}, label={below:0}] (92) at (12, -5) {};
		\node [style={type_vertex}, label={below:1}] (93) at (14, -5) {};
		\node [style={type_vertex}, label={below:2}] (94) at (6, -7) {};
		\node [style={type_vertex}, label={below:2}] (95) at (0, -7) {};
		\node [style=vertex] (96) at (14, -7) {};
		\node [style={type_vertex}, label={below:3}] (97) at (21.5, 1) {};
		\node [style={type_vertex}, label={below:2}] (98) at (21, 2.5) {};
		\node [style={type_vertex}, label={below:1}] (99) at (19, 2.5) {};
		\node [style={type_vertex}, label={below:0}] (100) at (18.5, 1) {};
		\node [style=vertex] (101) at (20, 0) {};
		\node [style=none] (102) at (20, -0.5) {};
		\node [style=none] (103) at (19.5, 0) {};
		\node [style=none] (104) at (22, 1) {};
		\node [style=none] (105) at (21, 3) {};
		\node [style=none] (106) at (21.5, 2.5) {};
		\node [style=none] (107) at (20.5, 2.5) {};
		\node [style={type_vertex}, label={below:3}] (108) at (21.5, -6) {};
		\node [style={type_vertex}, label={below:2}] (109) at (21, -4.5) {};
		\node [style={type_vertex}, label={below:1}] (110) at (19, -4.5) {};
		\node [style={type_vertex}, label={below:0}] (111) at (18.5, -6) {};
		\node [style=vertex] (112) at (20, -7) {};
		\node [style=none] (113) at (18, -6) {};
		\node [style=none] (114) at (18.5, -6.5) {};
		\node [style=none] (115) at (18.5, -4.25) {};
		\node [style=none] (116) at (21.5, -4.5) {};
		\node [style=none] (117) at (21, -4) {};
		\node [style=none] (118) at (21, -5) {};
		\node [style=none] (119) at (13, -1.5) {$O^{T_4}_a$};
		\node [style=none] (120) at (20, -1.5) {$O^{T_4}_b$};
		\node [style=none] (121) at (1, -8.5) {$M^{P_3}_a$};
		\node [style=none] (122) at (7, -8.5) {$M^{P_3}_b$};
		\node [style=none] (123) at (13, -8.5) {$M^{P_3}_c$};
		\node [style=none] (124) at (20, -8.5) {$N^{Q_4}$};
	\end{pgfonlayer}
	\begin{pgfonlayer}{edgelayer}
		\draw [style=edge, in=180, out=-90, looseness=1.25] (16.center) to (17.center);
		\draw [style=edge] (17.center) to (18.center);
		\draw [style=edge, in=-90, out=0] (18.center) to (19.center);
		\draw [style=edge] (19.center) to (20.center);
		\draw [style=edge, in=0, out=90, looseness=1.25] (20.center) to (21.center);
		\draw [style=edge, in=90, out=-180, looseness=0.25] (21.center) to (16.center);
		\draw [style=edge, in=0, out=-90, looseness=1.25] (44.center) to (45.center);
		\draw [style=edge] (45.center) to (46.center);
		\draw [style=edge, in=-90, out=180] (46.center) to (47.center);
		\draw [style=edge] (47.center) to (48.center);
		\draw [style=edge, in=180, out=90, looseness=1.25] (48.center) to (49.center);
		\draw [style=edge, in=90, out=0, looseness=0.25] (49.center) to (44.center);
		\draw [in=-90, out=180, looseness=0.50] (56.center) to (58.center);
		\draw [style=edge, in=-75, out=0, looseness=1.25] (56.center) to (57.center);
		\draw [style=edge, in=0, out=105, looseness=0.75] (61.center) to (59.center);
		\draw [style=edge, in=60, out=-180] (59.center) to (60.center);
		\draw [style=edge, in=90, out=-120, looseness=0.75] (60.center) to (58.center);
		\draw [style=edge] (61.center) to (57.center);
		\draw [style=edge, in=180, out=90, looseness=1.25] (66.center) to (67.center);
		\draw [style=edge] (67.center) to (68.center);
		\draw [style=edge, in=90, out=0] (68.center) to (69.center);
		\draw [style=edge] (69.center) to (70.center);
		\draw [style=edge, in=0, out=-90, looseness=1.25] (70.center) to (71.center);
		\draw [style=edge, in=-90, out=180, looseness=0.25] (71.center) to (66.center);
		\draw [style=edge, in=0, out=-90, looseness=1.25] (76.center) to (77.center);
		\draw [style=edge] (77.center) to (78.center);
		\draw [style=edge, in=-90, out=180] (78.center) to (79.center);
		\draw [style=edge] (79.center) to (80.center);
		\draw [style=edge, in=180, out=90, looseness=1.25] (80.center) to (81.center);
		\draw [style=edge, in=90, out=0, looseness=0.25] (81.center) to (76.center);
		\draw [style=edge, in=-180, out=-90, looseness=1.25] (86.center) to (87.center);
		\draw [style=edge] (87.center) to (88.center);
		\draw [style=edge, in=-90, out=0] (88.center) to (89.center);
		\draw [style=edge] (89.center) to (90.center);
		\draw [style=edge, in=0, out=90, looseness=1.25] (90.center) to (91.center);
		\draw [style=edge, in=90, out=-180, looseness=0.25] (91.center) to (86.center);
		\draw [in=-90, out=0, looseness=0.50] (102.center) to (104.center);
		\draw [style=edge, in=-105, out=180, looseness=1.25] (102.center) to (103.center);
		\draw [style=edge, in=180, out=75, looseness=0.75] (107.center) to (105.center);
		\draw [style=edge, in=120, out=0] (105.center) to (106.center);
		\draw [style=edge, in=90, out=-60, looseness=0.75] (106.center) to (104.center);
		\draw [style=edge] (107.center) to (103.center);
		\draw [in=-135, out=90, looseness=0.75] (113.center) to (115.center);
		\draw [style=edge, in=-150, out=-90] (113.center) to (114.center);
		\draw [style=edge, in=-90, out=30, looseness=0.75] (118.center) to (116.center);
		\draw [style=edge, in=-15, out=90, looseness=0.75] (116.center) to (117.center);
		\draw [style=edge, in=45, out=150, looseness=0.50] (117.center) to (115.center);
		\draw [style=edge] (118.center) to (114.center);
	\end{pgfonlayer}
\end{tikzpicture}

Then the following inequality holds on $E_5$-free hypergraphs:
\begin{equation}\label{pi345_equation}
    \begin{split}
        0 \leq \ \ \frac{2}{3} & \left\llbracket \left( E_3^{P_1} - \frac{3}{4} P_1 \right)^2 \right\rrbracket_{P_1} + 
        \frac{1}{6} \left\llbracket \left( L^{P_2}_a - L^{P_2}_b \right)^2 \right\rrbracket_{P_2} + \\
        \frac{13}{12} & \left\llbracket \left( M^{P_3}_a + M^{P_3}_b + M^{P_3}_c - \frac{1}{2} P_3 \right)^2 \right\rrbracket_{P_3} + 
        \frac{11}{12} \left\llbracket \left( E_4^{P_3} - \frac{1}{2} P_3 \right)^2 \right\rrbracket_{P_3} + \\
        2 & \left\llbracket \left( N^{Q_4} - \frac{1}{2} Q_4 \right)^2 \right\rrbracket_{Q_4} +
        \frac{1}{2} \left\llbracket \left( O^{T_4}_a - O^{T_4}_b \right)^2 \right\rrbracket_{P_4} \leq \quad \frac{3}{8} - E_4.
    \end{split}
\end{equation}

So far the only verification of \cref{pi345_equation} requires a tedious (computer assisted) checking of all the $2102$ hypergraphs in $\mathcal{H}^{(3)}_6$ without $E_5$. This can be found in the supplement. The corresponding lower bound is attained at $G_n = K^{(3)}_{\lfloor n/2 \rfloor} \bigsqcup K^{(3)}_{\lceil n/2 \rceil}$. Note that $d(E_5, G_n)=0$ while $\lim_{n \rightarrow \infty} d(E_4, G_n) = \frac{3}{8}$.

\section{Concluding Remarks}\label{outro_section}

This paper investigated a natural extension of the generalized Turán problem to hypergraphs. The result matches the best-known general bounds for $k$-graphs but fails to provide tight bounds when $k>3$.

The main combinatorial insight comes from the simple inequality \begin{equation}\label{flag_square_equation}
    0 \leq \left\llbracket \left( K_m^{T_{m-1}} - xT_{m-1} \right)^2 \right\rrbracket_{T_{m-1}},
\end{equation} combined with a close approximation of $\left\llbracket \left( K_m^{T_{m-1}} \right)^2 \right\rrbracket_{T_{m-1}}$. As shown in the paper, the convex combination of these squares includes difficult results for the generalized hypergraph Turán problem. 

The long list of questions improved by the plain flag algebraic method indicates that finding more sophisticated squares can greatly improve the available density bounds. It would be interesting to identify other families of simple linear density relations (like the one described in \cref{main_lemma}) whose conic combination includes new bounds for extremal hypergraph problems, even better if the bounds are tight. The provided certificate for $\pi\left(K^{(3)}_4, K^{(3)}_5\right) = 3/8$ can perhaps be generalized to larger cases. It is interesting that for the smallest $k, g, r$ tuple, which is not already known ($k=2$) and is not a classical hypergraph Turán problem ($k=g$), the exact solution follows from flag algebraic calculations. It also highlights the limitations of computer assisted searches: calculations for problems with higher parameters are infeasible.

\subsection{Finding Squares}

There is an easy to describe reason why \cref{flag_square_equation} fails to provide tight bounds for $k$-graphs where $k>2$ but is asymptotically exact when $k=2$. The extremal configuration for $\pi\left(n, K^{(2)}_g, K^{(2)}_r\right)$ is a unique balanced $(r-1)$-partite graph, call it $G_r(n)$, and say $G_r^{T_{m-1}}(n)$ is the same structure with a complete $(m-1)$-tuple marked as a type. Any choice of $T_{m-1}$ results in the same $\lim_{n \rightarrow \infty} d\left(K^{T_{m-1}}_m, G_r^{T_{m-1}}(n)\right)$ value, which is $x^{(2)}_{m, r} = 1 - \frac{m-1}{r-1}$. In contrast, the conjectured optimal constructions when $k>2$ give different values for different $T_{m-1}$ choices, therefore no $x \in \mathbb{R}$ exists with $$\left\llbracket \left( K_m^{T_{m-1}} - xT_{m-1} \right)^2 \right\rrbracket_{T_{m-1}} = 0$$ on the conjectured optimal constructions. This slackness gives the difference between the conjectured optimal constructions and the proved bounds here. The values $x^{(k)}_{m, r}$ are chosen optimally, any asymptotically significant improvement must utilize a different combinatorial insight.

\bibliographystyle{alpha}
\bibliography{readings.bib}
\end{document}